\documentclass[12pt,twoside]{amsart}
        \usepackage {amssymb,latexsym,amsthm,amsmath}
        \usepackage{hyperref}
        \usepackage[utf8]{inputenc}
        \usepackage{color}
        
        \usepackage{palatino}
        
\long\def\symbolfootnote[#1]#2{\begingroup%
\def\thefootnote{\fnsymbol{footnote}}\footnote[#1]{#2}\endgroup}

\makeatletter
\def\imod#1{\allowbreak\mkern10mu({\operator@font mod}\,\,#1)}
\makeatother
\def \B {{\rm B}}
\newtheorem{theorem}{Theorem}[section]
\newtheorem{lemma}[theorem]{Lemma}
\newtheorem{corollary}[theorem]{Corollary}
\newtheorem{proposition}[theorem]{Proposition}
\newtheorem*{theorem*}{Theorem}
\theoremstyle{definition}
\newtheorem{remark}[theorem]{Remark}

\newtheorem{problem}[theorem]{Problem}
\newtheorem{example}[theorem]{Example}

\numberwithin{equation}{section}

\newcommand{\ignore}[1]{}

\newcommand{\mynote}[1]{}
\begin{document}
\setcounter{section}{0}
\setcounter{tocdepth}{1}
\title[Centralizer classes in the group of upper triangular matrices]{Conjugacy classes of centralizers in the group of upper triangular matrices}
\author{Sushil Bhunia}
\address{IISER Mohali, Knowledge city, Sector 81, SAS Nagar, P.O. Manauli, Punjab 140306, INDIA}
\email{sushilbhunia@gmail.com}
\date{}
\thanks{The author is supported by the SERB, India (No. PDF/2017/001049) and DST-RFBR joint Indo-Russian project (No. INT/RUS/RFBR/P-288).}
\today
%%%%%%%%%%%%%%%%%%%%%%%%%%%%%%%%%%%%%%%%%%%%%%%%%%%%%%%%%
\begin{abstract}
	Let $G$ be a group. Two elements $x,y \in G$ are said to be in the same $z$-class if their centralizers in $G$ are conjugate within $G$. In this paper, we prove that the number of $z$-classes in the group of upper triangular matrices is infinite provided that the field is infinite and size of the matrices is at least $6$, and finite otherwise.
\end{abstract}
\maketitle
\section{Introduction}
Let $G$ be a group. Two elements $x$ and $y$ in $G$ are said to be $z$-equivalent, denoted by
$x\sim_{z} y$, if their centralizers in $G$ are conjugate, i.e., $\mathcal Z_G(y)=g\mathcal Z_G(x)g^{-1}$ for some $g\in G$, where $\mathcal Z_G(x):=\{y\in G \mid xy=yx \}$ denotes centralizer of $x$ in $G$. Clearly, $\sim_{z}$ is an equivalence relation on $G$. The equivalence classes with respect to this relation are called \textbf{$z$-classes}. It is easy to see that if two elements of a group $G$ are conjugate then their centralizers are conjugate thus they are also $z$-equivalent. However, in general, the converse is not true. In geometry, $z$-classes describe the behaviour of dynamical types (see for example~\cite{Ku1}, \cite{Ci}, and \cite{Go}). That is, if a group $G$ is acting on a manifold $M$ then understanding (dynamical types of) orbits is related to understanding (conjugacy classes of) centralizers. 

Robert Steinberg~\cite{St} (Section 3.6 Corollary 1 to Theorem 2) proved that for a reductive algebraic group $G$ defined over an algebraically closed field, of good characteristic, the number of $z$-classes is finite. A natural question that followed:
{\it Is the number of $z$-classes finite for algebraic group $G$ defined over an arbitrary field $k$?}
In~\cite{Si}, A. Singh studied $z$-classes for real compact groups of type $\mathrm{G}_2$.
Ravi S. Kulkarni, in~\cite{Ku}, proved that the number of $z$-classes in $\mathrm{GL}_n(k)$ is finite if the field $k$ has only finitely many field extensions of any fixed finite degree. Unless otherwise specified, we will always assume that $k$ is a field of $\mathrm{char}\; \neq 2$.
Let $V$ be an $n$-dimensional vector space over a field 
$k$ equipped with a non-degenerate symmetric or skew-symmetric bilinear form $B$. Then, in~\cite{GK}, 
it is proved that there are only finitely many $z$-classes in orthogonal groups $\mathrm{O}_n(k)$ and symplectic groups $\mathrm{Sp}_n(k)$ if $k$ has only finitely many field extensions of any fixed finite degree.
Let $k$ be a perfect field with a non-trivial Galois automorphism of order $2$. 
Let $V$ be an $n$-dimensional vector space over $k$ equipped with a non-degenerate Hermitian form $H$. 
Suppose that the fixed field $k_0$ has only finitely many field extensions of any fixed finite degree.
Then, in~\cite{BS}, we proved that the number of $z$-classes in the unitary group $\mathrm{U}_n(k_0)$ is finite. All the groups mentioned so far are special types of reductive algebraic groups. A natural problem to study would be to consider $z$-classes in non-reductive ones. In particular, one may consider the following problem:
\begin{problem}\label{prob2}
	Is the number of $z$-classes finite for a solvable  algebraic group?
\end{problem} 
 Let $G$ be a connected solvable linear algebraic group over an algebraically closed field $k$, then by Lie-Kolchin theorem (see Theorem 17.6~\cite{Hu1}) $G$ is a subgroup of the group of upper triangular matrices in $\mathrm{GL}_n(k)$ for some $n$. 
 
 Let $\mathrm{B}_n(k)$ denote the group of upper triangular matrices in $\mathrm{GL}_n(k)$. 
In this paper, we solve the Problem~\ref{prob2} for this special classes of groups.
In a sequel, we will do this for general nilpotent and solvable groups. The main result of this paper is the following theorem, which solves Problem~\ref{prob2} for the group of upper triangular matrices:
\begin{theorem}\label{maintheorem}\noindent
\begin{enumerate}
\item For $2\leq n \leq 5$, the number of $z$-classes in $\mathrm{B}_n(k)$ is finite. 
\item For $n\geq 6$, the number of $z$-classes in $\mathrm{B}_n(k)$ is infinite.
\end{enumerate}
\end{theorem}
In Section 2, we explore the semisimple $z$-classes for the group of upper triangular matrices. In Section 3, we study unipotent conjugacy classes and unipotent $z$-classes in $\mathrm{B}_n(k)$. In Section 4, we prove our main theorem of this paper. Throughout the paper, we assume that $k$ is an infinite field of $\mathrm{char}\; k\neq 2$. Here we include an appendix, which contains an explicit computation of unipotent conjugacy classes and their representatives in the group of upper triangular matrices $\B_n(k)$ for $n=2, 3, 4, 5$ using Belitskii's algorithm.
%%%%%%%%%%%%%%%%%%%%%%%%%%%%%%%%%%%%%%%%%%%%%%%%%%%%%%%
\section{Semisimple \texorpdfstring{$z$}{}-classes in \texorpdfstring{$\B_n(k)$}{}}
Let $n$ be a positive integer with a partition  
$\lambda=(1^{k_{1}} 2^{k_{2}}\ldots n^{k_{n}})$, denoted by 
$\lambda \vdash n$, 
i.e., $n=\sum_{i}ik_{i}$. 
\begin{proposition}\label{semisimplez}
	The number of semisimple $z$-classes in $\B_n(k)$ is 
	\[ \displaystyle \sum_{(1^{k_1} 2^{k_2}\ldots n^{k_n})\vdash n}\frac{n!}{\prod_{j=1}^{n}(j!)^{k_j}(k_j!)}.\] So, in particular, the number of semisimple $z$-classes in $\B_n(k)$ is finite.
\end{proposition}
\begin{proof}
	Semisimple elements in $\B_n(k)$ are nothing but diagonals up to conjugacy (for details see the first-page second paragraph in \cite{Ro}). So the number of semisimple $z$-classes in $\B_n(k)$ is equal to the number of $z$-classes of diagonals in $\B_n(k)$. Now we give a combinatorial argument to count the $z$-classes of diagonals in $\B_n(k)$, which is as follows: 
	Let $\lambda=(1^{k_{1}} 2^{k_{2}}\ldots n^{k_{n}})$ be a partition of $n$. Let us consider the following multiset \[M:=\{\underbrace{a_{11}, a_{12},\ldots, a_{1k_1}}_{k_1}; \underbrace{a_{21}, a_{21}, \ldots, a_{2k_2}, a_{2k_2}}_{2k_2}; \ldots, \underbrace{a_{i1},\ldots, a_{i1}, \ldots, a_{ik_i}, \ldots, a_{ik_i}}_{ik_i}; \ldots\}.\]
	So the number of ways to order the above multiset is 
	\[\frac{n!}{(1!)^{k_1}(2!)^{k_2}\cdots (n!)^{k_n}}.\]
	Now look at the action of $S_{k_1}\times S_{k_2}\times \cdots \times S_{k_n}$ on the ordered tuples via 
	\[(\sigma_1\times \sigma_2\times \cdots \times \sigma_n) (\cdots a_{ij} \cdots)=(\cdots a_{i\sigma_i(j)}\cdots),\]
	where $\sigma_i \in S_{k_i}$ and $S_{k_i}$'s are the symmetric groups on $k_i$ symbols.
	Therefore the size of each orbit under the above action is 
	$\prod_{j=1}^{n}(k_j!)$ as the stabilizer is the identity group. 
	So the number of orbits is 
	\[\frac{n!}{\prod_{j=1}^{n}(j!)^{k_j}\prod_{j=1}^{n}(k_j!)}.\]
	Hence the result.
\end{proof}
A numerical example should make the above argument transparent.
\begin{example}
	Let $n=5$, then the number of partitions of $n$ is equal to $7$ and are given by $(5^1), (1^14^1), (2^13^1), (1^23^1), (1^12^2), (1^32^1), (1^5)$. Now  
		\begin{enumerate}
			\item For $\lambda=(5^1)\vdash 5$, the number of semisimple $z$- classes is $\frac{5!}{5!}=1$ and representative is the following: \[\mathrm{diag}\;(\alpha, \alpha, \alpha, \alpha, \alpha).\] 
			\item For $\lambda=(1^14^1)\vdash 5$, the number of semisimple $z$- classes is $\frac{5!}{4!}=5$ and representatives are given by the following: 
			\begin{align*}
			\mathrm{diag}\;(\alpha, \alpha, \alpha, \alpha, \beta);&\hspace{15mm}
			\mathrm{diag}\;(\alpha, \alpha, \alpha, \beta, \alpha);\\
			\mathrm{diag}\;(\alpha, \alpha, \beta, \alpha, \alpha);&\hspace{15mm}
			\mathrm{diag}\;(\alpha, \beta, \alpha, \alpha, \alpha);
			\end{align*}
			\vspace{-12mm}
			\[\mathrm{diag}\;(\beta, \alpha, \alpha, \alpha, \alpha).\]
			\item For $\lambda=(2^13^1)\vdash 5$, the number of semisimple $z$- classes is $\frac{5!}{3!2!}=10$ and representatives are given by the following: 
			\begin{align*}
			\mathrm{diag}\;(\alpha, \alpha, \alpha, \beta, \beta);&\hspace{15mm}
			\mathrm{diag}\;(\alpha, \alpha, \beta, \alpha, \beta);\\
			\mathrm{diag}\;(\alpha, \beta, \alpha, \alpha, \beta);&\hspace{15mm}
			\mathrm{diag}\;(\beta, \alpha, \alpha, \alpha, \beta);\\
			\mathrm{diag}\;(\beta, \alpha, \alpha, \beta, \alpha);&\hspace{15mm}
			\mathrm{diag}\;(\beta, \alpha, \beta, \alpha, \alpha);\\
			\mathrm{diag}\;(\beta, \beta, \alpha, \alpha, \alpha);&\hspace{15mm}
			\mathrm{diag}\;(\alpha, \alpha, \beta, \beta, \alpha);\\
			\mathrm{diag}\;(\alpha, \beta, \beta, \alpha, \alpha);&\hspace{15mm}
			\mathrm{diag}\;(\alpha, \beta, \alpha, \beta, \alpha).
			\end{align*}
			\item For $\lambda=(1^23^1)\vdash 5$, the number of semisimple $z$-classes is $\frac{5!}{(1!)^22!3!}=10$ and  representatives are given by the following: 
			\begin{align*}
			\mathrm{diag}\;(\alpha, \alpha, \alpha, \beta, \gamma); &\hspace{15mm}
			\mathrm{diag}\;(\alpha, \alpha, \beta, \alpha, \gamma); \\
			\mathrm{diag}\;(\alpha, \beta, \alpha, \alpha, \gamma); &\hspace{15mm}
			\mathrm{diag}\;(\beta, \alpha, \alpha, \alpha, \gamma); \\
			\mathrm{diag}\;(\beta, \alpha, \alpha, \gamma, \alpha); &\hspace{15mm}
			\mathrm{diag}\;(\beta, \alpha, \gamma, \alpha, \alpha); \\
			\mathrm{diag}\;(\alpha, \beta, \gamma, \alpha, \alpha); &\hspace{15mm}
			\mathrm{diag}\;(\alpha, \alpha, \beta, \gamma, \alpha); \\
			\mathrm{diag}\;(\alpha, \beta, \alpha, \gamma, \alpha); &\hspace{15mm}
			\mathrm{diag}\;(\beta, \gamma, \alpha, \alpha, \alpha).
			\end{align*}
			\item For $\lambda=(1^12^2)\vdash 5$, the number of semisimple $z$-classes is $\frac{5!}{(2!)^2 2!}=15$ and representatives are given as follows:
			\begin{align*}
			\mathrm{diag}\;(\alpha, \alpha, \beta, \beta, \gamma); &\hspace{10mm}
			\mathrm{diag}\;(\alpha, \beta, \alpha, \beta, \gamma); &
			\mathrm{diag}\;(\beta, \alpha, \alpha, \beta, \gamma);\\
			\mathrm{diag}\;(\beta, \alpha, \beta, \gamma, \alpha); &\hspace{10mm}
			\mathrm{diag}\;(\alpha, \beta, \beta, \gamma, \alpha); &
			\mathrm{diag}\;(\gamma, \alpha, \alpha, \beta, \beta); \\
			\mathrm{diag}\;(\alpha, \alpha, \beta, \gamma, \beta); &\hspace{10mm}
			\mathrm{diag}\;(\alpha, \alpha, \gamma, \beta, \beta); &
			\mathrm{diag}\;(\alpha, \beta, \gamma, \alpha, \beta); \\
			\mathrm{diag}\;(\alpha, \beta, \gamma, \beta, \alpha); &\hspace{10mm}
			\mathrm{diag}\;(\alpha, \gamma, \alpha, \beta, \beta); &
			\mathrm{diag}\;(\alpha, \gamma, \beta, \alpha, \beta); \\
			\mathrm{diag}\;(\alpha, \gamma, \beta, \beta, \alpha); &\hspace{10mm}
			\mathrm{diag}\;(\gamma, \alpha, \beta, \alpha, \beta); &\;
			\mathrm{diag}\;(\gamma, \beta, \alpha, \alpha, \beta).
			\end{align*}
			\item For $\lambda=(1^32^1)\vdash 5$, the number of semisimple $z$-classes is $\frac{5!}{(1!)^33!2!}=10$ and representatives are given by the following: 
			\begin{align*}
			\mathrm{diag}\;(\alpha, \alpha, \beta, \gamma, \delta); &\hspace{15mm}
			\mathrm{diag}\;(\beta, \gamma, \delta, \alpha, \alpha);\\
			\mathrm{diag}\;(\beta, \gamma, \alpha, \delta, \alpha);&\hspace{15mm}
			\mathrm{diag}\;(\beta, \gamma, \alpha, \alpha, \delta);\\
			\mathrm{diag}\;(\beta, \alpha, \gamma, \alpha, \delta);&\hspace{15mm}
			\mathrm{diag}\;(\beta, \alpha, \alpha, \gamma, \delta);\\
			\mathrm{diag}\;(\alpha, \beta, \alpha, \gamma, \delta);&\hspace{15mm}
			\mathrm{diag}\;(\alpha, \beta, \gamma, \alpha, \delta);\\
			\mathrm{diag}\;(\alpha, \beta, \gamma, \delta, \alpha);&\hspace{15mm}
			\mathrm{diag}\;(\beta, \alpha, \gamma, \delta, \alpha).
			\end{align*}
			\item For $\lambda=(1^5)\vdash 5$, the number of semisimple $z$-classes is  $\frac{5!}{(1!)^55!}=1$ and the representative is given as follows: \[\mathrm{diag}\;(\alpha, \beta, \gamma, \delta, \eta).\]	
		\end{enumerate}
		Therefore the total number of semisimple $z$-classes in $\B_5(k)$ is $52$.
	\end{example}
%%%%%%%%%%%%%%%%%%%%%%%%%%%%%%%%%%%%%%%%%%%%%%%%%%%
\section{Unipotent \texorpdfstring{$z$}{}-classes in \texorpdfstring{$\B_n(k)$}{}}
Let $\alpha\in k$, define 
\[x_{\alpha}=\begin{pmatrix}
0&1&\alpha&0&0&0\\&0&0&0&1&0\\&&0&1&1&0\\
&&&0&0&0\\&&&&0&1\\&&&&&0\end{pmatrix}.\] 
Let $u_{\alpha}=I_6+x_{\alpha}\in \B_6(k)$ be a unipotent element.
The following result is already known. We record this result as we are going to use it.
\begin{proposition}\label{unicong}\noindent
	\begin{enumerate}
		\item For $2\leq n \leq 5$, the number of unipotent conjugacy classes in $\B_n(k)$ is finite. In particular, the numbers are $2, 5, 16, 60$ for $n=2, 3, 4, 5$ respectively.
		\item For $n\geq 6$, the number of unipotent conjugacy classes in $\B_n(k)$ is infinite.
	\end{enumerate}
\end{proposition}
\begin{proof}
	\begin{enumerate}
		\item Use Belitskii's algorithm, for details see~\cite{CXLF},~\cite{Ko} and~\cite{Th}. Also, see the appendix for explicit calculations of the number of unipotent conjugacy classes. We also give representatives of the unipotent conjugacy classes in $\B_n(k)$ for $2\leq n\leq 5$.
		\item The proof was originally given by M. Roitman in~\cite{Ro} for $n=12$. Then latter Djokovic and Malzan,  in~\cite{DM}, proved this for $n=6$ and in fact, it is the minimum value for which this happens to be true by part (1). For completeness, we will give this prove again.
		It is enough to prove this for $n=6$. Let $u_{\alpha}$ and $u_{\beta}$ are conjugate in $\B_6(k)$, i.e., 
		$Pu_{\alpha}P^{-1}=u_{\beta}$ for some $P\in \B_6(k)$. Then 
		$Px_{\alpha}=x_{\beta}P$. Let $P=(p_{ij})$, then we get the following:
		\begin{align*}
		p_{23}&= p_{45}=0\\
		p_{11}&= p_{22}\\
		p_{22}&= p_{55}\\ 
		p_{55}&= p_{33}\\
		\alpha p_{11}&= p_{23}+\beta p_{33}.
		\end{align*}
		Therefore from the above equation, we get $\alpha=\beta$. So the number of unipotent conjugacy classes in $\B_6(k)$ is infinite as $k$ is an infinite field. Hence it is true for $\B_n(k)$ provided $n\geq 6$.
	\end{enumerate}
\end{proof}
\begin{corollary}\label{unipotentz1}
	For $2\leq n \leq 5$, the number of unipotent $z$-classes in $\B_n(k)$ is finite.
\end{corollary}
\begin{proof}
	If two elements are conjugate, then they are also $z$-conjugate. So this follows from the first part of Proposition~\ref{unicong}.
\end{proof}
Now the centralizer of $u_{\alpha}$,  $\mathcal{Z}_{\B_6(k)}(u_{\alpha})$ is the following:
\[\left\{ \begin{pmatrix}
a&b_1&b_2&b_3&b_4&b_5\\&a&0&b_2-\alpha b_6&b_1+b_2-\alpha b_7&b_4-\alpha b_8\\
&&a&b_6&b_7&b_8\\&&&a&0&(\alpha+1)b_7-b_1-b_2\\
&&&&a&b_1+b_2-\alpha b_7\\
&&&&&a\end{pmatrix}\mid a\in k^{\times}, b_i\in k \right\}.\]
\begin{lemma}\label{unipotentz2}
	For $n\geq 6$, the number of unipotent $z$-classes in $\B_n(k)$ is infinite.
\end{lemma}
\begin{proof}
	It is enough to prove this lemma for $n=6$. Now assume that $n=6$. Suppose that $u_{\alpha}$ and $u_{\beta}$ are $z$-conjugate, then $P\mathcal{Z}_{\B_6(k)}(u_{\alpha})P^{-1}=
	\mathcal{Z}_{\B_6(k)}(u_{\beta})$ for some $P\in \B_{6}(k)$. 
	
	\noindent
	\underline{\textbf{Claim}:} $\alpha=\beta$. 
	
	Now $PAP^{-1}\in \mathcal{Z}_{\B_6(k)}(u_{\beta})$ for all $A\in \mathcal{Z}_{\B_6(k)}(u_{\alpha})$. So $PAP^{-1}=A'$ for some $A' \in \mathcal{Z}_{\B_6(k)}(u_{\beta})$. Observe that two upper triangular matrices are conjugate via a upper triangular matrix implies that they have the same diagonal entries. Let $P=(p_{ij})\in \B_6(k)$ and $A, A'$ have the form described as above. Then we get 
	\begin{align}
	a'&= a\\
	b_1'&= b_1p_{11}p_{22}^{-1}\\
	b_2'&= (b_2p_{11}-b_1p_{11}p_{23}p_{22}^{-1})p_{33}^{-1}\\
	b_6'&= b_6p_{33}p_{44}^{-1}\\
	b_7'&= (b_7p_{33}-b_6p_{33}p_{45}p_{44}^{-1})p_{55}^{-1}.
	\end{align}
	\begin{align}\label{eq6}
	(b_2-\alpha b_6)p_{22}+b_6p_{23}&= (b_2'-\beta b_6')p_{44}.
	\end{align}
	\begin{align}\label{eq7}
	((\alpha+1)b_7-b_1-b_2)p_{44}+(b_1+b_2-\alpha b_7)p_{45}&= 
	((\beta+1)b_7'-b_1'-b_2')p_{66}.
	\end{align}
	From Equation (\ref{eq6}), we get, $p_{23}=0$ and 
	$\alpha p_{22}=\beta p_{33}$.
	
	(If $b_6=b_2=0$ and $b_1\neq 0,$ then 
	$b_6'=0$ and 
	$b_2'=-b_1p_{11}p_{23}p_{22}^{-1}p_{33}^{-1}$. So \[-b_1p_{11}p_{23}p_{22}^{-1}p_{33}^{-1}p_{44}=0,\] hence $p_{23}=0$, since $b_1\neq 0$ and $p_{ii}\neq 0$. And if 
	$b_1=b_2=0$ and $b_6\neq 0$, then $b_1'=0=b_2'$ and $ b_6'= b_6p_{33}p_{44}^{-1}$. So \[-\alpha b_6p_{22}+b_6p_{23}=-\beta b_6p_{44}p_{33}p_{44}^{-1}.\] Now, since $p_{23}=0$, we get $\alpha b_6p_{22}=\beta b_6a_{33}$, which implies $\alpha p_{22}=\beta p_{33}$, as $b_6\neq 0$). 
	
	From Equation (\ref{eq7}), we get, $p_{22}=p_{33}$.
	
	(If $b_1=b_6=b_7=0$ and $ b_2\neq 0,$ then $b_1'=b_6'=b_7'=0$ and 
	$ b_2'=b_2p_{11}p_{33}^{-1}$. So from Equation (\ref{eq6}) we get \[p_{44}-p_{45}=p_{11}p_{66}p_{33}^{-1},\] since $b_2\neq 0$. Again if $b_2=b_6=b_7=0$ and $ b_1\neq 0$, then $b_6'=b_7'=0$ and 
	$b_1'= b_1p_{11}p_{22}^{-1};  b_2'=-b_1p_{11}p_{23}p_{22}^{-1}p_{33}^{-1}=0$, as $p_{23}=0$. 
	So again from Equation (\ref{eq6}) we get, 
	\[p_{44}-p_{45}=p_{11}p_{66}p_{22}^{-1},\] since $b_1\neq 0$. 
	Therefore from the above two we get $p_{22}=p_{33}$). 
	Hence $\alpha=\beta$. Therefore the result is true for $n=6$. Now for $n >6$. Let \[U_{\alpha}:=\left(\begin{array}{c|c}
	u_{\alpha}&0\\
	\hline
	0&I_{n-6}
	\end{array}\right),\, U_{\beta}:=\left(\begin{array}{c|c}
	u_{\beta}&0\\
	\hline
	0&I_{n-6}
	\end{array}\right)\in \B_n(k)\] 
	be two unipotent elements which are $z$-conjugate in $\B_n(k)$. Then \[Q\mathcal{Z}_{\B_n(k)}(U_{\alpha})Q^{-1}=\mathcal{Z}_{\B_n(k)}(U_{\beta})\]
	for some $Q\in \B_n(k)$. Therefore $QCQ^{-1}=C'$ for some $C\in \mathcal{Z}_{\B_n(k)}(U_{\alpha})$ and $C' \in \mathcal{Z}_{\B_n(k)}(U_{\beta})$. Now write $Q, C$ and $C'$ in block form, we get 
	\[\left(\begin{array}{c|c}
	P&*\\
	\hline
	0&*
	\end{array}\right)
	\left(\begin{array}{c|c}
	A&*\\
	\hline
	0&*
	\end{array}\right)=\left(\begin{array}{c|c}
	A'&*\\
	\hline
	0&*
	\end{array}\right)
	\left(\begin{array}{c|c}
	P&*\\
	\hline
	0&*
	\end{array}\right)\]
	for some $P, A, A'\in \B_6(k)$. Hence $PA=A'P$, which reduces to the case of $n=6$. Therefore the number of unipotent $z$-classes in $\B_n(k)$ ($n\geq 6$) is infinite.
\end{proof}
\begin{corollary}
The unipotent $z$-classes for $\B_6(k)$ is parametrized by elements of the field $k$.
\end{corollary}
\begin{proof}
Let $u$ be a unipotent element of $\B_6(k)$. Then using Belitskii's algorithm (see appendix and~\cite{Ko} for details) we get $bub^{-1}=u_{\alpha}$ for some $b\in \B_6(k)$ and for some $\alpha \in k$, where $u_{\alpha}$ is defined at the beginning of this section. Again from Lemma~\ref{unipotentz2} we have $u_{\alpha}\sim_{z} u_{\beta}$ if and only if $\alpha=\beta$, where $\alpha, \beta \in k$. Therefore unipotent $z$-classes in $\B_6(k)$ are completely determined by the elements of $k$ via the map $\alpha \mapsto u_{\alpha}$.
\end{proof}
%%%%%%%%%%%%%%%%%%%%%%%%%%%%%%%%%%%%%%%%%%%%%%%%%%%%%%
\section{\texorpdfstring{$z$}{}-classes in \texorpdfstring{$\B_n(k)$}{}}
\begin{lemma}\label{techlemma}
	Let $g\in G$ and $g=g_sg_u$ be the Jordan decomposition of $g$, and $\alpha \in G$. Then we have $$\alpha \mathcal{Z}_{\mathcal{Z}_G(g_s)}(g_u)\alpha^{-1}
	=\mathcal{Z}_{\alpha \mathcal{Z}_G(g_s)\alpha^{-1}}(\alpha g_u\alpha^{-1}).$$
\end{lemma}
\begin{proof}
	Let $x\in \mathcal{Z}_{\mathcal{Z}_G(g_s)}(g_u)$ then $xg_s=g_sx$ and $xg_u=g_ux$. Therefore $\alpha xg_u\alpha^{-1}=\alpha g_u x\alpha^{-1}$ implies that $(\alpha x\alpha^{-1})(\alpha g_u \alpha^{-1})=(\alpha g_u \alpha^{-1})(\alpha x\alpha^{-1})$. Therefore $\alpha x\alpha^{-1}\in \mathcal{Z}_{\alpha \mathcal{Z}_G(g_s)\alpha^{-1}}(\alpha g_u\alpha^{-1})$.
	
	On the other hand let $y\in \mathcal{Z}_{\alpha \mathcal{Z}_G(g_s)\alpha^{-1}}(\alpha g_u\alpha^{-1})$, then $(\alpha^{-1}y \alpha)g_s=g_s(\alpha^{-1}y\alpha)$ and $y(\alpha g_u\alpha^{-1})=(\alpha g_u\alpha^{-1})y$. Now the last equation is same as $(\alpha^{-1}y\alpha)g_u=g_u(\alpha^{-1}y\alpha)$. So $\mathcal{Z}_{\alpha \mathcal{Z}_G(g_s)\alpha^{-1}}(\alpha g_u\alpha^{-1})\subseteq \alpha \mathcal{Z}_{\mathcal{Z}_G(g_s)}(g_u)\alpha^{-1}$. Hence the result.
\end{proof}
\begin{remark}\label{remark}
	Let us assume that the number of semisimple $z$-classes in $G$ is $n$ and representatives are given by $s_1, s_2, \ldots, s_n$. Let $g\in G$ then $g=g_sg_u$ is the Jordan decomposition of $g$. By the above assumption, $g_s$ will be $z$-conjugate to $s_i$ for some $i=1, 2, \ldots, n$. Without loss of generality, say $g_s\sim_z s_1$, i.e., $\alpha \mathcal{Z}_G(g_s)\alpha^{-1}=\mathcal{Z}_G(s_1)$ for some $\alpha \in G$. Then 
	\[\alpha \mathcal{Z}_G(g)\alpha^{-1}=\alpha \mathcal{Z}_{\mathcal{Z}_G(g_s)}(g_u)\alpha^{-1}
	=\mathcal{Z}_{\mathcal{Z}_G(s_1)}(\alpha g_u\alpha^{-1}).\] 
	The first equality follows from the uniqueness of the Jordan decomposition, i.e., \[\mathcal{Z}_G(g)=\mathcal{Z}_G(g_s)\cap \mathcal{Z}_G(g_u)=\mathcal{Z}_{\mathcal{Z}_G(g_s)}(g_u),\] and the second equality follows from Lemma \ref{techlemma}. 
	Now if the number of unipotent $z$-classes in $\mathcal{Z}_G(s_i)$ is finite for all $i=1,2,\ldots, n$. Then the number of $z$-classes in $G$ is finite. So the upshot is the following:
	
	If we know that the number of semisimple $z$-classes in $G$ is finite, and the number of unipotent $z$-classes in centralizer of semisimple elements is finite, then the 
  number of $z$-classes in $G$ is finite.
\end{remark}
\subsection{Proof of the Theorem \ref{maintheorem}:}
\begin{enumerate}
\item The number of semisimple $z$-classes in $\B_n(k)$  is finite follows from Proposition \ref{semisimplez}. The number of unipotent $z$-classes in $\B_n(k)$ is finite for $2\leq n\leq 5$ follows from Corollary \ref{unipotentz1}. So the number of $z$-classes in $\B_n(k)$, for $2\leq n \leq 5$, is finite follows from Lemma \ref{techlemma} (see also Remark \ref{remark}).
\item The number of unipotent $z$-classes in $\B_n(k)$ is infinite for $n\geq 6$ follows from Lemma \ref{unipotentz2}. Therefore the number of $z$-classes in $\B_n(k)$ is infinite provided $n\geq 6$.
\end{enumerate}
%%%%%%%%%%%%%%%%%%%%%%%%%%%%%%%%%%%%%%%%%%%%%%%%%%%%%%%%%
\section{Appendix}
Two matrices $A, B \in \B_n(k)$ are said to be \textbf{conjugate} if $B=PAP^{-1}$ for some $P\in \B_n(k)$. Here we are following \cite{Ko}.

\noindent
\underline{\textbf{Belitskii's Algorithm for $\B_n(k)$}:} 

\noindent
Let $A=(a_{ij})\in \B_n(k)$. Elements of the matrix $A$ are ordered by \[a_{nn};a_{n-1 n-1}, a_{n-1 n};\ldots; a_{11}, a_{12}, \ldots, a_{1n},\] i.e., a sequence from bottom to top and in each row from left to right. 

\noindent
\underline{\textbf{AIM}:} The aim of this algorithm is to simplify the first entry in the above sequence, then the second entry and so on. By ``simplifying" we mean replacing the entry by $0$ or $1$ (conjugating the matrix $A$ by upper triangular matrices) if possible. If not, then we continue with the next entry in the above sequence. At each step, we take care not to disturb any of the reductions obtained so far. 

Let $e_{ij}(\alpha)$ be an elementary matrix, 
with $(i,j)^{\mathrm{th}}$ element equal to $\alpha$ and $0$ everywhere else. 
Two matrices $A$ and $ B$ are conjugate if and only if one reduces to the other by a sequence of the following two elementary transformations: 
\begin{enumerate}
	\item Multiply row $i$ by $\alpha \neq 0$, then multiply column $i$ by $\alpha^{-1}$; the elementary transformations can be obtained as $A\mapsto PAP^{-1}$, $P=I+e_{ii}(\alpha -1)$.
	\item For $i<j$, multiply $I+e_{ij}(\alpha)$ from the left; then multiply $I+e_{ij}(-\alpha)$ from the right; the elementary transformations can be obtained as $A\mapsto PAP^{-1}$, $P=I+e_{ij}(\alpha)$.
\end{enumerate}
\vspace{10mm}
\underline{\textbf{Algorithm for $\B_n(k)\; (2\leq n \leq 5)$:}}

\noindent
\textbf{Step 1:} Let $a_{pq}$ be the first unreduced entry of $A$. If the column of $a_{pq}$ contains an entry $a_{iq}=1$ located under $a_{pq}$ (i.e., $i>p$) and $a_{iq}$ is the first nonzero entry in the row, then $a_{pq}=0$ by transformation of the type (2) with $P=I+e_{pi}(-a_{pq})$.

\noindent
\textbf{Step 2:} Suppose that the column of $a_{pq}$ does not contain such an entry $a_{iq}$. If $a_{pq}$ is the first nonzero entry of that row, then $a_{pq}=1$ by transformation of the type (1) with $P=I+e_{pp}(a_{pq}^{-1}-1)$.

\noindent
\textbf{Step 3:} Suppose $a_{pr}=1$ is the first nonzero entry in the row of $a_{pq}$, then $a_{pq}=0$ by the transformation of the type (2) with $P=I+e_{rq}(a_{pq})$. But this might disturb the row $a_{r*}$, which was reduced before. However, this does not happen if the row $a_{q*}$ is zero.

\noindent
\textbf{Step 4:} If $a_{ri}=1=a_{qj}$ are the first nonzero entries of corresponding rows and $i<j$, then the above transformation by $P=I+e_{rq}(a_{pq})$ disturbs the row of $a_{ri}$, which can be restored by the transformation of the type (2) with $P=I+e_{ij}(a_{pq})$  (this transformation does not disturb the reduced entries since the matrix is $5\times 5$ this need not be true for $6\times 6$ matrices).
In this case, we get $a_{pq}=0$.  

\noindent
\textbf{Step 5:} If $a_{qj}$ is the first nonzero entry of the row and $a_{ri}=0$ for all $i<j$, then $a_{pq}=1$ by transformation of the type (1) with 
$P=I+e_{qq}(a_{pq}-1)$. But this transformation disturbs row of $a_{qj}$ by changing $a_{qj}=1$ into $a_{qj}:=a_{pq}$. This can be restored by a transformation of type (1) with 
$P=I+e_{jj}(a_{pq}-1)$. The latter transformation does not disturb already reduced entries if $j^{\text{th}}$ row is zero. If $j^{\text{th}}$ row is not zero, then we have $j=4$ since the dimension is $5$ and the element $a_{45}=1$ was changed to $a_{45}:=a_{pq}$. This can be restored by a transformation of the type (1) with $P=I+e_{55}(a_{pq}-1)$.

Here we use Belitskii's algorithm (for details see~\cite{Ko}) for unipotent elements. 
Number of unipotent conjugacy classes in $\B_5(k)$ is $60$ and representatives are the following:
\[\begin{pmatrix}1&1&&&\\&1&&&\\&&1&&\\
&&&1&\\&&&&1\end{pmatrix}
\begin{pmatrix}1&&1&&\\&1&&&\\&&1&&\\
&&&1&\\&&&&1\end{pmatrix}
\begin{pmatrix}1&&&1&\\&1&&&\\&&1&&\\
&&&1&\\&&&&1\end{pmatrix}
\begin{pmatrix}1&&&&1\\&1&&&\\&&1&&\\
&&&1&\\&&&&1\end{pmatrix}\]
\[\begin{pmatrix}1&&&&\\&1&1&&\\&&1&&\\
&&&1&\\&&&&1\end{pmatrix}
\begin{pmatrix}1&&&&\\&1&&1&\\&&1&&\\
&&&1&\\&&&&1\end{pmatrix}
\begin{pmatrix}1&&&&\\&1&&&1\\&&1&&\\
&&&1&\\&&&&1\end{pmatrix}
\begin{pmatrix}1&&&&\\&1&&&\\&&1&1&\\
&&&1&\\&&&&1\end{pmatrix}\]
\[\begin{pmatrix}1&&&&\\&1&&&\\&&1&&1\\
&&&1&\\&&&&1\end{pmatrix}
\begin{pmatrix}1&&&&\\&1&&&\\&&1&&\\
&&&1&1\\&&&&1\end{pmatrix};
\begin{pmatrix}1&&&&\\&1&&&\\&&1&1&\\
&&&1&1\\&&&&1\end{pmatrix} 
\begin{pmatrix}1&&&&\\&1&1&&\\&&1&&\\
&&&1&1\\&&&&1\end{pmatrix} \]
\[\begin{pmatrix}1&&&&\\&1&&1&\\&&1&&\\
&&&1&1\\&&&&1\end{pmatrix}
\begin{pmatrix}1&1&&&\\&1&&&\\&&1&&\\
&&&1&1\\&&&&1\end{pmatrix}
\begin{pmatrix}1&&1&&\\&1&&&\\&&1&&\\
&&&1&1\\&&&&1\end{pmatrix} 
\begin{pmatrix}1&&&1&\\&1&&&\\&&1&&\\
&&&1&1\\&&&&1\end{pmatrix}\]
\[\begin{pmatrix}1&&&&\\&1&&&1\\&&1&1&\\
&&&1&\\&&&&1\end{pmatrix}
\begin{pmatrix}1&&&&\\&1&1&&\\&&1&1&\\
&&&1&\\&&&&1\end{pmatrix}
\begin{pmatrix}1&1&&&\\&1&&&\\&&1&1&\\
&&&1&\\&&&&1\end{pmatrix}
\begin{pmatrix}1&&1&&\\&1&&&\\&&1&1&\\
&&&1&\\&&&&1\end{pmatrix} \]
\[\begin{pmatrix}1&&&1&\\&1&&&\\&&1&&1\\
&&&1&\\&&&&1\end{pmatrix}
\begin{pmatrix}1&&&&1\\&1&&&\\&&1&1&\\
&&&1&\\&&&&1\end{pmatrix}
\begin{pmatrix}1&&&&\\&1&&1&\\&&1&&1\\
&&&1&\\&&&&1\end{pmatrix} 
\begin{pmatrix}1&1&&&\\&1&&&\\&&1&&1\\
&&&1&\\&&&&1\end{pmatrix}\] 
\[\begin{pmatrix}1&&1&&\\&1&&&\\&&1&&1\\
&&&1&\\&&&&1\end{pmatrix}
\begin{pmatrix}1&&&1&\\&1&1&&\\&&1&&\\
&&&1&\\&&&&1\end{pmatrix} 
\begin{pmatrix}1&&&&1\\&1&1&&\\&&1&&\\
&&&1&\\&&&&1\end{pmatrix} 
\begin{pmatrix}1&&&&\\&1&1&&\\&&1&&1\\
&&&1&\\&&&&1\end{pmatrix}\]
\[\begin{pmatrix}1&1&&&\\&1&1&&\\&&1&&\\
&&&1&\\&&&&1\end{pmatrix}
\begin{pmatrix}1&&1&&\\&1&&1&\\&&1&&\\
&&&1&\\&&&&1\end{pmatrix}
\begin{pmatrix}1&&&1&\\&1&&&1\\&&1&&\\
&&&1&\\&&&&1\end{pmatrix}
\begin{pmatrix}1&1&&&\\&1&&1&\\&&1&&\\
&&&1&\\&&&&1\end{pmatrix}\] 
\[\begin{pmatrix}1&&1&&\\&1&&&1\\&&1&&\\
&&&1&\\&&&&1\end{pmatrix}
\begin{pmatrix}1&1&&&\\&1&&&1\\&&1&&\\
&&&1&\\&&&&1\end{pmatrix}
\begin{pmatrix}1&&&&1\\&1&&1&\\&&1&&\\
&&&1&\\&&&&1\end{pmatrix};
\begin{pmatrix}1&&&&\\&1&1&&\\&&1&1&\\
&&&1&1\\&&&&1\end{pmatrix} \]
\[\begin{pmatrix}1&1&&&\\&1&&&\\&&1&1&\\
&&&1&1\\&&&&1\end{pmatrix}
\begin{pmatrix}1&&1&&\\&1&&&\\&&1&1&\\
&&&1&1\\&&&&1\end{pmatrix} 
\begin{pmatrix}1&&&&\\&1&1&1&\\&&1&&\\
&&&1&1\\&&&&1\end{pmatrix}
\begin{pmatrix}1&&&1&\\&1&1&&\\&&1&&\\
&&&1&1\\&&&&1\end{pmatrix} \]
\[\begin{pmatrix}1&1&&&\\&1&1&&\\&&1&&\\
&&&1&1\\&&&&1\end{pmatrix}
\begin{pmatrix}1&1&&&\\&1&&1&\\&&1&&\\
&&&1&1\\&&&&1\end{pmatrix} 
\begin{pmatrix}1&&1&&\\&1&&1&\\&&1&&\\
&&&1&1\\&&&&1\end{pmatrix}
\begin{pmatrix}1&1&&1&\\&1&&&\\&&1&&\\
&&&1&1\\&&&&1\end{pmatrix} \]
\[\begin{pmatrix}1&&1&1&\\&1&&&\\&&1&&\\
&&&1&1\\&&&&1\end{pmatrix}
\begin{pmatrix}1&1&&&\\&1&&&1\\&&1&1&\\
&&&1&\\&&&&1\end{pmatrix} 
\begin{pmatrix}1&1&&&\\&1&1&&\\&&1&1&\\
&&&1&\\&&&&1\end{pmatrix}
\begin{pmatrix}1&&1&&\\&1&&&1\\&&1&1&\\
&&&1&\\&&&&1\end{pmatrix} \]
\[\begin{pmatrix}1&1&1&&\\&1&&&\\&&1&1&\\
&&&1&\\&&&&1\end{pmatrix}
\begin{pmatrix}1&&&1&\\&1&1&&\\&&1&&1\\
&&&1&\\&&&&1\end{pmatrix}
\begin{pmatrix}1&&&&1\\&1&1&&\\&&1&1&\\
&&&1&\\&&&&1\end{pmatrix}
\begin{pmatrix}1&1&&&\\&1&1&&\\&&1&&1\\
&&&1&\\&&&&1\end{pmatrix} \]
\[\begin{pmatrix}1&&1&&\\&1&&1&\\&&1&&1\\
&&&1&\\&&&&1\end{pmatrix}
\begin{pmatrix}1&1&1&&\\&1&&&\\&&1&&1\\
&&&1&\\&&&&1\end{pmatrix};
\begin{pmatrix}1&1&&&\\&1&1&&\\&&1&1&\\
&&&1&1\\&&&&1\end{pmatrix} 
\begin{pmatrix}1&1&1&&\\&1&&&\\&&1&1&\\
&&&1&1\\&&&&1\end{pmatrix} \]
\[\begin{pmatrix}1&1&&&\\&1&1&1&\\&&1&&\\
&&&1&1\\&&&&1\end{pmatrix}
\begin{pmatrix}1&&&1&\\&1&1&1&\\&&1&&\\
&&&1&1\\&&&&1\end{pmatrix}
\begin{pmatrix}1&1&1&&\\&1&&&1\\&&1&1&\\
&&&1&\\&&&&1\end{pmatrix};
\begin{pmatrix}1&&&&\\&1&&&\\&&1&&\\
&&&1&\\&&&&1\end{pmatrix}.\]
Therefore we have obtained also all unipotent conjugacy classes for $\B_n(k)$ for $n=2,3,4$. In particular, all we have to do is to look for the bottom right $2\times 2$, $3\times 3$ and $4\times 4$ corners of the above $5\times 5$ case.

For $n=2$, representatives are the following: 
$\begin{pmatrix}1&1\\&1\end{pmatrix} ; \begin{pmatrix}1&\\&1\end{pmatrix}$. 

For $n=3$, representatives are the following: 
\[\begin{pmatrix}1&1&\\&1&\\&&1\end{pmatrix}
\begin{pmatrix}1&&1\\&1&\\&&1\end{pmatrix}
\begin{pmatrix}1&&\\&1&1\\&&1\end{pmatrix};
\begin{pmatrix}1&1&\\&1&1\\&&1\end{pmatrix} ;
\begin{pmatrix}1&&\\&1&\\&&1\end{pmatrix}.\]
For $n=4$, representatives are the following:
\[\begin{pmatrix}1&1&&\\&1&&\\&&1&\\&&&1\end{pmatrix}
\begin{pmatrix}1&&1&\\&1&&\\&&1&\\&&&1\end{pmatrix}
\begin{pmatrix}1&&&1\\&1&&\\&&1&\\&&&1\end{pmatrix}
\begin{pmatrix}1&&&\\&1&1&\\&&1&\\&&&1\end{pmatrix}\]
\[\begin{pmatrix}1&&&\\&1&&1\\&&1&\\&&&1\end{pmatrix}
\begin{pmatrix}1&&&\\&1&&\\&&1&1\\&&&1\end{pmatrix};
\begin{pmatrix}1&1&&\\&1&1&\\&&1&\\&&&1\end{pmatrix}
\begin{pmatrix}1&1&&\\&1&&1\\&&1&\\&&&1\end{pmatrix}\]
\[\begin{pmatrix}1&1&&\\&1&&\\&&1&1\\&&&1\end{pmatrix}
\begin{pmatrix}1&&1&\\&1&&1\\&&1&\\&&&1\end{pmatrix}
\begin{pmatrix}1&&1&\\&1&&\\&&1&1\\&&&1\end{pmatrix}
\begin{pmatrix}1&&&\\&1&1&\\&&1&1\\&&&1\end{pmatrix}\]
\[\begin{pmatrix}1&&&1\\&1&1&\\&&1&\\&&&1\end{pmatrix};
\begin{pmatrix}1&1&&\\&1&1&\\&&1&1\\&&&1\end{pmatrix}
\begin{pmatrix}1&1&1&\\&1&&\\&&1&1\\&&&1\end{pmatrix};
\begin{pmatrix}1&&&\\&1&&\\&&1&\\&&&1\end{pmatrix}.\]

\vspace{7mm}
\textbf{Acknowledgement:} The author would like to acknowledge Dr. Anupam Singh and Dr. Rohit Joshi of IISER Pune for many helpful discussions. The author thanks Dr. Pranab Sardar and Dr. Krishnendu Gongopadhyay of IISER Mohali for encouragement.
%%%%%%%%%%%%%%%%%%%%%%%%%%%%%%%%%%%%%%%%%%%%%%%%%%%%

\end{document}